\documentclass{amsart}

\usepackage[utf8]{inputenc}
\usepackage{amssymb,amsbsy,amsthm,amsmath,graphicx,epsfig,enumitem}
\usepackage[dvipsnames]{xcolor}

\usepackage[backend=biber,maxbibnames=99, style=alphabetic, doi=false, isbn=false, eprint=false]{biblatex}
\usepackage{hyphenat}

\let\oldproof\proof
\def\proof{\oldproof\unskip}

\newtheorem{thm}{Theorem}[section]
\newtheorem*{thm*}{Theorem}
\newtheorem{theorem}{Theorem}[section]

\newtheorem{lemma}[thm]{Lemma}
\newtheorem{prop}[thm]{Proposition}
\newtheorem{proposition}[thm]{Proposition}
\newtheorem{cor}[thm]{Corollary}

\theoremstyle{definition}
\newtheorem{dfn}[thm]{Definition}

\theoremstyle{remark}
\newtheorem*{ex}{Example}
\newtheorem*{remark}{Remark}

\newcommand{\Nat}{\mathbb{N}}

\newcommand{\Real}{\mathbb{R}}
\newcommand{\Torus}{\mathbb{T}}
\newcommand{\Int}{\mathbb{Z}}

\renewcommand{\d}{\mathrm{d}}

\newcommand{\id}{\mathrm{id}}

\renewcommand{\Re}{\operatorname{Re}}
\renewcommand{\Im}{\operatorname{Im}}

\DeclareMathOperator{\Leb}{L}

\DeclareMathOperator{\Ind}{\mathbf{1}}
\DeclareMathOperator{\supp}{supp}
\newcommand{\flim}[1]{\operatorname{\text{$#1$-}\!\lim}}
\newcommand{\setb}[2]{\left\{\,#1\,\middle|\,#2\,\right\}}
\newcommand{\setl}[1]{\{\,#1\,\}}
\newcommand{\sing}[1]{\{#1\}}

\title{On Wiener's lemma for locally compact abelian groups}

\author{Eike Schulte}
\address{Institute of Mathematics, University of Leipzig,
P.O. Box 100 920, 04009 Leipzig, Germany}
\email{schulte@math.uni-leipzig.de}


\begin{document}

\begin{abstract}
Inspired by an extension of Wiener’s lemma on the relation of measures $\mu$ on the unit circle and their Fourier coefficients $\widehat{\mu}(k_n)$ along subsequences $(k_n)$ of the natural numbers by Cuny, Eisner and Farkas \cite{cuny_eisner_farkas_2019}, we study the validity of the lemma when the Fourier coefficients are weighted by a sequence of probability measures. By using convergence with respect to a filter derived from these measure sequences, we obtain similar results, now also allowing the consideration of locally compact abelian groups other than $\Torus$ and $\Real$. As an application, we present an extension of a result of Goldstein \cite{goldstein_1996} on the action of semigroups on Hilbert spaces.
\end{abstract}

\maketitle

\section{Introduction}

In \cite{cuny_eisner_farkas_2019}, Cuny, Eisner and Farkas consider a variant of a classical lemma of Wiener: For a complex measure $\mu$ on the unit circle $\Torus$, this lemma relates Cesàro sums of the Fourier coefficients $\widehat{\mu}(n)$ (defined for $n \in \Int$) to the values of $\mu$ on atoms (see Theorem \ref{thm:classic-wiener} below). In their variant, they replace the summation over all values $\widehat{\mu}(n)$ by summation only along some “good” subsequence $(k_n)$. Now the Fourier transform is defined on any locally compact abelian group, but because the dual group need not be discrete nor naturally ordered, there need not be a natural sequence of elements to sum along. To circumvent this problem, we consider sequences of (probability) measures (called “measure sequences” for short) on our group instead and define what it means for such a sequence to be “good”. 

In dealing with Cesàro summation it is often useful to be able to switch between convergence of these sums and a “sum-less” convergence, where, for example, products can be more easily treated. This is classically achieved by the Koopman-von Neumann lemma, relating convergence of Cesàro sums to “convergence in density”. A density function can be defined from our measure sequence similar to how the usual density of a set $J \subseteq \Nat$ can be defined as the limit of $\frac{1}{N}(\delta_1 + \dots + \delta_N)(J)$. However, convergence in density also relies on the natural order of $\Int$ (or rather, $\Nat$); we therefore have to replace it. This is achieved by considering convergence with respect to some filter, suitably defined from our new density (Proposition \ref{prop:koopman-von-neumann}).

Equipped with these constructions we see that Wiener’s lemma holds in a similar form as before (Theorem \ref{thm:my-wiener}). We can also consider “extremality” (in a similar sense to \cite{cuny_eisner_farkas_2019}) of measures and see that being extremal for a sufficiently good measure sequence implies being a Dirac measure.

Finally, we apply our lemma to obtain a version of a result of Goldstein about operators on Hilbert spaces. The classical case corresponds to the group $\Int$ (or rather, its subsemigroup $\Nat$), and a version for continuous one-parameter semigroups is given in \cite{cuny_eisner_farkas_2019}. Using our lemma, we get a version for a locally compact abelian group equipped with a suitable measure sequence and a sufficiently big subsemigroup (Theorem \ref{prop:semigroup-goldstein}).

The rest of the paper is organized as follows: In the remainder of this section, we will introduce basic notation used throughout the paper. In Section~2 we define measure sequences and prove the appropriate version of the Koopman-von Neumann lemma. In Section~3 we explain what it means for a measure sequence to be “good” and show the generalization of Wiener’s lemma for these sequences. In Section~4 we consider extremality and show that Dirac measures are characterized by the integrals over the Fourier coefficients with respect to suffeciently good measure sequences. Finally, in Section~5 we apply the earlier results to actions of semigroups and groups on Hilbert spaces and obtain a version of a theorem of Goldstein \cite{goldstein_1996}.

\subsection*{Acknowlegdements.}

The author thanks Tanja Eisner for bringing this subject to his attention and for reading a preliminary version of the paper.

\subsection*{Notation}

Let $G$ be a locally compact abelian group. We fix a Haar measure on $G$ and denote integration with respect to this measure by
\[
\int_G f(g)\,\mathrm{d}g.
\]
If $H$ is isomorphic to the Pontryagin dual of $G$ (hence also locally compact and abelian) via the isomorphism $\psi: H \to \widehat{G}$, we write $\langle g, h\rangle := \psi(h)(g)$ and we say that $(G, H)$ is a \emph{pair of dual groups}. The two most classical pairs of groups are $(\mathbb{T}, \mathbb{Z})$ (where $\mathbb{T}$ is the unit circle) with $\langle \lambda, n \rangle = \lambda^{-n}$ and $(\mathbb{R}, \mathbb{R})$ with $\langle r, s \rangle = \mathrm{e}^{-2\pi \mathrm{i} rs}$. Whenever such a pairing occurs, we will assume that the Haar measures on $G$ and $H$ are chosen such that the Fourier transform $\Leb^2(G) \to \Leb^2(H)$, $f \mapsto \widehat{f}$, given on the subspace $\Leb^1(G) \cap \Leb^2(G)$ by
\[
\widehat{f}(h) = \int_G \langle g, h \rangle f(g) \,\mathrm{d}g,
\]
is isometric. We will also use the Fourier transform
\[
\widehat{\mu}(h) = \int_G \langle g, h \rangle \mu(\mathrm{d}g)
\]
of complex measures $\mu$ on $G$.

We will occasionally use filter convergence. If $\mathcal{F}$ is a filter on $I$ and $f : I \to X$ is a function into some topological space, we write $\text{$\mathcal{F}$-}\lim_i f(i)$ for the limit of the image filter $i_*\mathcal{F}$ on $X$. In other words, $\text{$\mathcal{F}$-}\lim_i f(i)$ denotes the element $x \in X$ such that for every neighborhood $U$ of $x$ there exists a set $F \in \mathcal{F}$ with $f(i) \in U$ for every $i \in F$ (provided such an element $x$ exists and is unique).

\section{Measure sequences}

\begin{dfn}\label{def:measure-sequence}
A \emph{measure sequence} $\nu = (\nu_N)_{N \in \mathbb{N}}$ on some measurable space $X$ is a sequence of probability measures $\nu_N$ on $X$. If no confusion can arise, we will denote the $\Leb^2$-norm with respect to $\nu_N$ by $\Vert \cdot \Vert_N$.
\end{dfn}

Note that the sets of integrable functions can vary with $N$ but that they always contain all bounded measurable functions on $X$ because the measures are finite.

\begin{ex}
The following measure sequences allow to recover the prior results mentioned in the introduction:
\begin{enumerate}[label=\arabic*.]
\item On $\mathbb{N}$ with the power set as the sigma algebra, we have sequence of measures $\mathrm{u} = (\mathrm{u}_N)$ where
\[
\mathrm{u}_N(A) = \frac{\#(A \cap \setl{0, \dots, N - 1})}{N}.
\]
Integration against these measures gives the partial Cesàro sums. We will also consider $\mathrm{u}$ as a measure sequence on $\Int$.

\item Generalizing 1., for any locally compact group $G$ and Følner sequence $(F_N)$ on $G$ we consider
\[
    \nu_N(A) := \frac{\mathrm{meas}(A \cap F_N)}{\mathrm{meas}(F_N)}.
\]
In particular, we will denote by $\mathrm{u}_N$ the sequence corresponding to the Følner sequence $([0,N])_{N \in \Nat}$ on $\Real$.

\item For every measure sequence $\nu = (\nu_n)$ we can form its \emph{Cesàro transform} $\mathrm{C}\nu$ where
\[
    (\mathrm{C}\nu)_N = \frac 1 N \sum_{n=1}^{N} \nu_n.
\]
In particular, the measure sequence $\mathrm{u}$ on $\Nat$ is the Cesàro transform of $(\delta_{N-1})_N$.
\end{enumerate}
\end{ex}

\begin{dfn}
For a measure sequence $\nu$ on $X$ and a measurable subset $J \subseteq X$ the~\emph{density} of $J$ (with respect to $\nu$) is defined by
\[
    \mathrm{d}_\nu(J) := \lim_{N \to \infty} \nu_N(J)
\]
if this limit exists.
\end{dfn}

For example, using the measure sequence $\mathrm{u}_N$ on $\mathbb{N}$ we recover the usual density of subsets of $\mathbb{N}$. Using the sequence $(\delta_N)$ (or more generally, $(\delta_{k_N})$), a set has density $1$ if and only if it contains every $N \in \Nat$ greater than some $N_0$ (or all $k_N$ where $N > N_0$) and density $0$ iff it contains no $N$ greater than some $N_0$ (none of the elements $k_N$ where $N > N_0$). Otherwise, the density of a set is undefined for these sequences.

By the usual argument that convergence implies convergence of the Cesàro sums, the density function $\operatorname{d}_{\operatorname{C}\nu}$ for the Cesàro transform of a measure sequence $\nu$ is an extension of $\operatorname{d}_\nu$.

The convergence of Cesàro sums is related to “convergence in density” via the Koopman--von Neumann lemma \cite[Lemma 9.16]{eisner_farkas_haase_nagel_2015}. To obtain a similar result for measure sequences, we introduce the filter $\mathcal{F}_\nu$ generated by all subsets of $X$ with $\nu$-density~$1$.

\begin{prop}[Koopman--von Neumann for measure sequences]\label{prop:koopman-von-neumann}
Let $X$ be a measurable space, $\nu$ a measure sequence on $X$ and $f : X \to [0, \infty[$ a bounded measurable function. The following are equivalent:
\[
\text{(i)}~\lim_{N \to \infty} \int_X f(x) \nu_N(\mathrm{d}x) = 0, \qquad \text{(ii)}~\text{$\mathcal{F}_\nu$-}\lim_x f(x) = 0.
\]
\end{prop}

\begin{proof}
(i $\implies$ ii) For each $\varepsilon > 0$ consider the set $L = \setb{ x \in X }{ f(x) \geq \varepsilon }$. It is measurable and
\[
\mathrm{d}_\nu(L) = \lim_{N \to \infty} \int_X \mathbf{1}_L(x) \nu_N(\mathrm{d}x) \leq \lim_{N \to \infty} \int_X \frac{1}{\varepsilon} f(x) \nu_N(\mathrm{d}x) = 0.
\]
Hence, $X \setminus L$ has density $1$ and $f(x) < \varepsilon$ on this set. Because $\varepsilon$ is arbitrary, this implies convergence with respect to the filter $\mathcal{F}_\nu$.

(ii $\implies$ i) Let $M$ be an upper bound for $f$ and $\varepsilon > 0$. By filter convergence, there exists $J \in \mathcal{F}_\nu$ such that $f(x) < \varepsilon / 2$ for all $x \in J$. By the definition of the filter $\mathcal{F}_\nu$, we may assume that $J$ is measurable with density $1$. The set $L := X \setminus J$ has density $0$, so we can find $N_0$ such that $\nu_N(L) < \varepsilon/(2M)$ for all $N \geq N_0$. For these $N$ we have
\[
\int_X f(x) \nu_N(\mathrm{d}x) = \int_J f(x) \nu_N(\mathrm{d}x) + \int_L f(x) \nu_N(\mathrm{d}x) \leq \int_J \frac{\varepsilon}{2} \nu_N(\mathrm{d}x) + M \nu_N(L) \leq \varepsilon,
\]
implying convergence to zero.
\end{proof}

\begin{remark}
“Convergence in density” of a sequence $(a_n)$ to some limit $a$ in a topological space $A$ means that there is a set $J \subseteq \Nat$ of (classical) density $1$ such that $\lim_{n \in J, n \to \infty} a_n = a$. It appears to be a stronger condition than the filter convergence used above, as a single set $J$ of density $1$ has to be chosen once and for all (depending only on the sequence $(a_n)$) which is then intersected with sets of the form $\setl{N, N+1, \dots}$ (where $N$ depends on the neighborhood of the limit $a$ under consideration), whereas filter convergence allows the choice of a set of density $1$ independently for each neighborhood. However, if $A$ is a bounded interval $[0, M]$ the two concepts agree after all: Indeed, the classical lemma and our version show that both of them are equivalent to the convergence of $\lim_{N \to \infty} \frac{1}{N} \sum_{n=1}^N a_n$.

In fact, convergence in density and convergence with respect the filter we associate to the standard density function are equivalent if the topological space $A$ is merely first countable. This can be seen by a close examination of the proof of the classical Koopman--von Neumann lemma (cf. the implication (i $\implies$ iii) in \cite[Lemma 9.16]{eisner_farkas_haase_nagel_2015}) -- note also that this what makes the classical lemma hard, the proof of our version is quite a bit more straightforward.
\end{remark}

\section{Wiener’s lemma}

We consider complex measures $\mu$, finite by definition. Wiener’s lemma is the following classical result.

\begin{theorem}[Wiener’s lemma]\label{thm:classic-wiener}
Let $\mu$ be a complex Borel measure on $\mathbb{T} = \setb{ z \in \mathbb{C} }{ \left|z\right| = 1 }$. Then,
\[
    \lim_{N \to \infty} \frac{1}{N} \sum_{n=1}^N \left|\widehat{\mu}(n)\right|^2 = \sum_{\text{$a$ atom}} \left|\mu\sing{a}\right|^2
\]
where the sum on the right is over all atoms of $\mu$ (i.e. points $a \in \mathbb{T}$ such that $\mu\sing{a} \neq 0$) and $\widehat{\mu}$ denotes the Fourier transform of $\mu$.
\end{theorem}

Inspired by the work of Cuny, Eisner and Farkas \cite{cuny_eisner_farkas_2019}, we will deduce a version of this lemma that uses the measure sequences introduced in the previous section.

\begin{dfn}\label{def:c-function}
Let $(G, H)$ be a pair of dual groups und $\nu$ a measure sequence on $H$. The sequence is called \emph{good} if the limit
\[
    \mathrm{c}_\nu(g) = \lim_{N \to \infty} \int_H \left< g, h \right> \nu_N(\mathrm{d}h)
\]
exists for every $g \in G$. In this case, the set $\setb{ g \in G }{ \mathrm{c}_\nu(g) \neq 0 }$ is called the \emph{spectrum} of $\nu$. If the measuare sequence is good and its spectrum is the set $\sing{ 1 }$, the sequence is called \emph{ergodic}.
\end{dfn}

Note that we always have $\mathrm{c}_\nu(1) = 1$ and $\mathrm{c}_\nu(g^{-1}) = \overline{\mathrm{c}_\nu(g)}$ because $\langle h, g^{-1} \rangle = \overline{\langle h, g \rangle}$. If the limit $\mathrm{c}_\nu(g)$ exists, its absolute value is at most $1$.

We now provide examples of good measure sequences.

\begin{proposition}
Let $(G, H)$ be a pair of dual groups and $\nu = (\nu_N)$ a measure sequence on $H$ derived from a Følner sequence $(F_N)$. Then $\nu$ is good.
\end{proposition}

\begin{proof}
Let $g \in G$. The dual pairing gives an action of $H$ on $\Torus$: $h \in H$ acts on $\lambda \in \Torus$ via multiplication by $\langle g, h \rangle$. Consider the identity function on $f = \id \in \Leb^2(\Torus)$. By the mean ergodic theorem for amenable groups \cite[Corollary 3.4]{greenleaf_1973}, the ergodic means $A_{F_N} f$ of $f$, given by
\[
    (A_{F_N} f)(\lambda) := \frac{1}{\operatorname{meas}(F_N)} \int_{F_N} \langle g, h \rangle \lambda \d h,
\]
converge to some function $\tilde f$ (in $\Leb^2$). Hence,
\begin{align*}
    \operatorname{c}_\nu(g)
        &= \lim_{N \to \infty} \frac 1 {\operatorname{meas}(F_N)} \int_{F_N} \langle g, h \rangle \d h
        = \lim_{N \to \infty} \int_\Torus \frac 1 {\operatorname{meas}(F_N)} \int_{F_N} \langle g, h \rangle \lambda \overline{\lambda} \d h \d \lambda \\
        &= \lim_{N \to \infty} \int_\Torus (A_{F_N}f)(\lambda) \overline{f(\lambda)} \d \lambda
        = \lim_{N \to \infty} (A_{F_N}f|f)
        = (\tilde f|f),
\end{align*}
in particular, the limit exists.
\end{proof}

\begin{lemma}
If $\nu$ is a good (or ergodic) measure sequence, so is its Cesàro transform $\mathrm{C}\nu$.
\end{lemma}

\begin{proof}
This is the usual argument that convergence implies Cesàro convergence.
\end{proof}

If $\mu$ and $\nu$ are measure sequences on groups $G$, $H$, we can obtain the measure sequence $\mu \otimes \nu$ on $G \times H$ by taking the product measure $\mu_N \otimes \nu_N$ for each $N$. By Fubini’s theorem, we have
\[
    \mathrm{c}_{\mu \otimes \nu}(g, h) = \mathrm{c}_\mu(g) \mathrm{c}_\nu(h),
\]
provided that both terms on the right-hand side exist. Hence the product sequence is good if both factors are.

\begin{lemma}\label{lem:c-filter-equivalent}
For $g \in G$, the following are equivalent:
\[
\text{(i)}~\left|\mathrm{c}_\nu(g)\right| = 1, \qquad
\text{(ii) the limit $\text{$\mathcal{F}_\nu$-}\lim_h \langle g, h \rangle$ exists}.
\]
If (i) and (ii) hold, the limit of (ii) is equal to $\mathrm{c}_\nu(g)$.
\end{lemma}

\begin{proof}
This follows from Proposition \ref{prop:koopman-von-neumann} by passing to the real parts and using the fact that complex numbers with absolute value at most $1$ converge to $1$ if and only if the real parts converge to $1$.
\end{proof}

\begin{prop}
The set $\Gamma_\nu := \setb{ g \in G }{ \left|c_\nu(g)\right| = 1 }$ is a subgroup of $G$ and $\mathrm{c}_\nu$ is a group homomorphism $\Gamma_\nu \to \Torus$.
\end{prop}

\begin{proof}
The fact that the product of two elements from $\Gamma_\nu$ lies in $g$ follows from the equivalent description in Lemma \ref{lem:c-filter-equivalent} because the product of two limits is the limit of the product by continuity of multiplication in $\Torus$. The rest of the assertion follows from the remarks following Definition \ref{def:c-function}.
\end{proof}

\begin{dfn}
We say that a measure sequence $\nu$ on a topological space $X$ \emph{goes to infinity} if $\lim_{N\to\infty} \nu_N(K) = 0$ for every compact set $K \subseteq X$.
\end{dfn}

\begin{lemma}\label{lem:almost-everywhere-zero}
Let $\nu$ be a good measure sequence that goes to infinity. Then $\mathrm{c}_\nu(g) = 0$ for almost every $g \in G$.
\end{lemma}

\begin{proof}
Let $\varphi \in \Leb^1(G)$ be a test function. Its Fourier transform $\widehat{\varphi}$ is a continuous function vanishing at infinity \cite[(31.5)]{hewitt_ross_1970}. Hence, 
\[
    \int_G \mathrm{c}_\nu(g) \varphi(g) \d g
        = \int_G \lim_{N \to \infty} \int_H \langle g, h \rangle \varphi(g) \nu_N(\d h) \d g
        = \lim_{N \to \infty} \int_H \widehat{\varphi}(h) \nu_N(\d h) = 0,
\]
where we use Fubini and dominated convergence for the second equality and that $\nu$ goes to infinity and $\widehat{\varphi}$ vanishes at infinity for the last one. As $\varphi$ is arbitrary, the claim follows.
\end{proof}

\begin{cor}
Let $\nu$ be a good measure sequence going to infinity on $H = \Int$ or $H = \Real$. Then $\Gamma_\nu$ is discrete.
\end{cor}

\begin{proof}
Subgroups of the relevant dual groups $\Torus$ and $\Real$ are either discrete or dense. As $\mathrm{c}_\nu$ is the pointwise limit of continuous functions, it is of Baire class $1$ and has points of continuity. At these points, $\mathrm{c}_\nu$ must be zero, for otherwise it would be non-zero on a non-empty open subset, contradicting Lemma \ref{lem:almost-everywhere-zero}. By continuity, the absolute value of $\mathrm{c}_\nu$ is less than $1$ on a neighborhood of these points. Thus $\Gamma_\nu$ cannot be dense.
\end{proof}

We are now ready to prove our first theorem: this is the promised generalization of Wiener’s lemma.

\begin{thm}\label{thm:my-wiener}
Let $(G, H)$ be a pair of dual groups and $\nu$ a measure sequence on $H$.
\begin{enumerate}
\item  Let $\mu$ be a complex Borel measure on $G$. Then
\[
    \lim_{N \to \infty} \left\Vert \widehat{\mu} \right\Vert_N^2 = \int_{G \times G} \mathrm{c}_\nu(g_0g_1^{-1}) (\mu \otimes \overline{\mu})(\d(g_0,g_1)).
\]
\item Let $\mu$ be a probability measure on $G$. If $\lim_{N \to \infty} \left\Vert\widehat{\mu}\right\Vert_N^2 = 1$ then the support of $\mu$ is contained in a coset of $\overline{\Gamma_\nu}$ and $\mathrm{c}_\nu(g_0 g_1^{-1}) = 1$ for every pair of $\mu$-atoms $g_0$, $g_1$.
\item Let $\mu$ be a discrete probability measure on $G$. If $\mathrm{c}_\nu(g_0 g_1^{-1}) = 1$ for every pair of $\mu$-atoms $g_0$, $g_1$ then
\[
    \lim_{N \to \infty} \left\Vert\widehat{\mu}\right\Vert_N^2 = 1.
\]
\item The measure sequence $\nu$ is ergodic if and only if
\[
    \lim_{N \to \infty} \left\Vert\widehat{\mu}\right\Vert_N^2 = \sum_{\text{$a$ atom}} \left|\mu\sing{a}\right|^2
\]
for every complex Borel measure $\mu$ on $G$.
\item Let $\nu$ be ergodic and $\mu$ be a probability measure on $G$. Then
\[
    \lim_{N \to \infty} \left\Vert\widehat{\mu}\right\Vert_N^2 = 1
\]
if and only if $\mu$ is a Dirac measure (i.e. $\mu = \delta_g$ for some $g \in G$).
\end{enumerate}
\end{thm}

Note that Wiener’s original lemma amounts to the claim that $\mathrm{u}$ (see the examples following Definition \ref{def:measure-sequence}) is an ergodic measure sequence. Unwinding definitions, this is just the well-known fact that $\lim_{N \to \infty} \frac 1 N \sum_{n=1}^N \lambda^{-n}$ (where $\lambda \in \Torus$) is $1$ if and only if $\lambda = 1$ and $0$ otherwise.

\begin{proof}
\begin{enumerate}
\item By Fubini’s theorem and the theorem on dominated convergence,
\begin{align*}
    \left\Vert\widehat{\mu}\right\Vert_N^2
        &= \int_H \int_G \left<g_0, h\right> \mu(\d g_0) \int_G \overline{\left<g_1, h\right>} \overline{\mu}(\d g_1) \nu_N(\d h) \\
        &= \int_{G \times G} \int_H \left<h, g_0g_1^{-1}\right> \nu_N(\d h) (\mu \otimes \overline{\mu})(\d(g_0, g_1))
\end{align*}
which converges to
\[
    \int_{G \times G} \mathrm{c}_\nu(g_0g_1^{-1})(\mu \otimes \overline{\mu})(\d(g_0,g_1))
\]
as $N \to \infty$.

\item If $\lim_{N \to \infty} \Vert\widehat{\mu}\Vert_N^2 = 1$ and $\mu$ is a probability measure, we conclude from (1) and $|\mathrm{c}_\nu(g_0g_1^{-1})| \leq 1$ that
\[
    1 = \lim_{N \to \infty} \Vert\widehat{\mu}\Vert_N^2 
    \leq \int_{G \times G} |\mathrm{c}_\nu(g_0g_1^{-1})| (\mu \otimes \overline{\mu})(\d(g_0,g_1)) 
    \leq 1,
\]
hence equality everywhere. Thus,
\[
    \int_G |\mathrm{c}_\nu(g_0g_1^{-1})| \mu(\d g_0) = 1
\]
for $\mu$-almost every $g_1 \in G$. Pick any such $g_1$.

We will now consider the function $f : g_0 \mapsto |\mathrm{c}_\nu(g_0g_1^{-1})|$ and show that $\supp \mu \subseteq \overline{\setb{g_0 \in G }{ f(g_0) = 1 }} = \overline{g_1 \Gamma_\nu}$. Assume that is not the case and choose $g \in \supp \mu$ and a neighborhood $U$ of $g$ such that $f(g_0) \neq 1$ (hence $f(g_0) < 1$) for every $g_0 \in U$. Because $g \in \supp \mu$, we have $\mu(U) > 0$. Then $\int_U f(g_0) \mu(\d g_0) < \mu(U)$ and hence \[
    \int_G f(g_0) \mu(\d g_0) < \mu(U) + \mu(G \setminus U) = 1.
\]
This is a contradiction, hence $\supp \mu$ is contained in the coset $\overline{g_1 \Gamma_\nu}$.

For the last part of (2), apply real parts to the formula from (1) instead. We get
\[
    \int_{G \times G} \Re \mathrm{c}_\nu(g_0g_1^{-1}) (\mu \otimes \overline{\mu})(\d(g_0,g_1)) = 1,
\]
hence
\[
    \int_{G} \Re \mathrm{c}_\nu(g_0g_1^{-1}) \mu(\d g_0) = 1
\]
for $\mu$-almost every $g_1$; in particular this holds for every $\mu$-atom. For every such atom we conclude $\Re \mathrm{c}_\nu(g_0 g_1^{-1}) = 1$ for almost all $g_0$, hence for all atoms $g_0$. As $|\mathrm{c}_\nu(g_0 g_1^{-1})| \leq 1$, we get $\mathrm{c}_\nu(g_0 g_1^{-1}) = 1$.

\item As $\mu$ is discrete, the integral in (1) is simply the sum over the atoms
\[
    \sum_{g_0, g_1} \mathrm{c}_\nu(g_0 g_1^{-1}) \mu\sing{g_0} \overline{\mu}\sing{g_1} = \sum_{g_0, g_1} \mu\sing{g_0} \mu\sing{g_1} = 1.
\]

\item If $\nu$ is ergodic, i.e. $\mathrm{c}_\nu = \Ind_{\sing{1}}$, we apply Fubini’s theorem and obtain
\begin{align*}
    \int_{G \times G} \mathrm{c}_\nu(g_0g_1^{-1})(\mu \otimes \overline{\mu})(\d(g_0,g_1))
    &= \int_G \int_G \Ind_{\sing{g_1}}(g_0) \mu(\d g_0) \overline{\mu}(\d g_1) \\
    &= \int_G \mu\sing{g_1} \overline{\mu}(\d g_1) \\
    &= \sum_a |\mu\sing{a}|^2.
\end{align*}

Conversely, if $\nu$ is not ergodic, choose $g \in G \setminus \sing{1}$ such that $\mathrm{c}_\nu(g) \neq 0$. If $\Re \mathrm{c}_\nu(g) \neq 0$ consider the probability measure $\mu = \frac{1}{2} (\delta_1 + \delta_g)$. Then
\begin{align*}
    & \int_{G \times G} \mathrm{c}_\nu(g_0g_1^{-1})(\mu \otimes \overline{\mu})(\d(g_0,g_1)) \\
    &\qquad    = \frac{1}{2} + \frac{1}{4}(\mathrm{c}_\nu(g) + \mathrm{c}_\nu(g^{-1})) \neq \frac{1}{2} = \sum_a |\mu\sing{a}|^2.
\end{align*}
If $\Im \mathrm{c}_\nu(g) \neq 0$ instead, procede similarly with $\mu = \frac{1}{2} (\delta_1 + \mathrm{i} \delta_g)$.

\item Finally, for a Dirac measure $\delta_a$, we have $\widehat{\delta_a}(h) = \langle a, h \rangle$. In particular, $|\widehat{\delta_a}(h)|^2 = 1$ for every $h \in H$, hence $\Vert \widehat{\delta_a} \Vert_N = 1$.

In the opposite direction, we apply (4) and get
\[
    1 = \lim_{N \to \infty} \Vert\widehat{\mu}\Vert_N^2 = \sum_a |\mu\sing{a}|^2 \leq \sum \mu\sing{a} \leq 1,
\]
hence equality. Because $0 \leq \mu\sing{a} \leq 1$ we get $(\mu\sing{a})^2 = \mu\sing{a}$ for every $a$, hence $\mu\sing{a} = 1$ for exactly one $a$ and $\mu$ is a Dirac measure. \qedhere
\end{enumerate}
\end{proof}

We mention briefly how to recover the result of Cuny, Eisner and Farkas from this: For a given good (or ergodic) sequence $(k_N)$ in $\Nat$, the Cesàro transform of the measure sequence $(\delta_{k_N})$ is a good (resp., ergodic) measure sequence $\nu$ for the pair of dual groups $(\Torus, \Int)$. The limit functions $c$ and $\mathrm{c}_\nu$ agree. Then parts (a), (b) and (c) of Proposition 2.6 in \cite{cuny_eisner_farkas_2019} correspond to parts (1), (4) and (5) above. Their Corollary 2.9 corresponds to parts (2) and (3).

\section{Extremality}

\begin{dfn}
Let $(G, H)$ be a pair of dual groups and $\nu$ a measure sequence on $H$.

\begin{enumerate}[label=(\alph*)]
\item We call a probability measure $\mu$ on $G$ \emph{extremal} for $\nu$ if $\lim_{N\to\infty} \Vert\widehat{\mu}\Vert_N^2 = 1$.

\item If every probability measure that is extremal for $\nu$ is a Dirac measure, the measure sequence $\nu$ is called \emph{extremal}.

\item If every discrete probability measure that is extremal for $\nu$ is a Dirac measure, the measure sequence $\nu$ is called \emph{extremal for discrete measures}.
\end{enumerate}
\end{dfn}

\begin{remark}
These terms are again inspired by Cuny, Eisner and Farkas \cite{cuny_eisner_farkas_2019}. Note however, that there is a minor but unfortunate mismatch between our terms: While a sequence $(k_n)$ is good in \cite{cuny_eisner_farkas_2019} iff the measure sequence $(\frac 1 N \sum_{n = 1}^N \delta_{k_n})$ is good, the sequence is \cite{cuny_eisner_farkas_2019}-extremal iff the measure sequence $(\delta_{k_N})$ is.
\end{remark}

\begin{thm}
Let $\nu$ be a measure sequence on $H$. For the assertions
\begin{enumerate}[label=(\roman*)]
\item $\nu$ is extremal,
\item $\nu$ is extremal for discrete measures,
\item $\mathrm{c}_\nu(g) = 1$ implies $g = 1$,
\end{enumerate}
we have the implications (i) $\implies$ (ii) $\iff$ (iii). Furthermore, (ii) $\implies$ (i) if $\nu$ is good and the subgroup $\overline{\Gamma_\nu}$ of $G$ is discrete and countable.
\end{thm}

\begin{proof}
(i $\implies$ ii) is trivial.

(ii $\implies$ iii) Suppose not. Choose $g \in G$ such that $g \neq 1$ but $\mathrm{c}_\nu(g) = 1$ and consider the (discrete) measure $\mu = \frac 1 2 (\delta_1 + \delta_g)$. We have
\[
    |\widehat{\mu}(h)|^2 = \frac 1 4 (2 + \langle g, h \rangle + \langle g^{-1}, h \rangle).
\]
Integrating and passing to the limit gives
\[
    \lim_{N\to\infty} \Vert\widehat{\mu}\Vert_N^2 = \frac 1 4 (2 + \mathrm{c}_\nu(g) + \mathrm{c}_\nu(g^{-1})) = 1
\]
even though $\mu$ is not Dirac, hence the sequence is not extremal for discrete measures.

(iii $\implies$ ii) Let $\mu$ be a discrete probability measure, i.e. $\mu = \sum_{i \in I} A_i \delta_{g_i}$ for some index set $I$ where $A_i > 0$, $\sum_{i} A_i = 1$ and $g_i$ are pairwise distinct elements of $G$, and assume that $\mu$ is extremal for $\nu$. We have
\[
    \widehat{\mu}(h) = \sum_{i \in I} A_i \langle g_i, h \rangle
\]
and hence
\[
    \Vert\mu\Vert_N^2 = \sum_{i} A_i^2 + 2 \sum_{i < j} A_i A_j \Re \int_H \langle g_ig_j^{-1},h \rangle \nu_N(\d h).
\]
By extremality, these terms converge to $1$ as $N \to \infty$. By an elementary computation, each $\Re \int_H \langle g_i g_j^{-1}, h \rangle \nu_N(\d h)$ converges to $1$. Because $\int_H \langle g_i g_j^{-1}, h \rangle \nu_N(\d h)$ is bounded in absolute value by $1$, these terms also converge to $1$, but their limit is just $\mathrm{c}_\nu(g_i g_j^{-1})$. By (iii), $g_i g_j^{-1} = 1$, so $I$ is a singleton set and $\mu$ is Dirac.

(ii $\implies$ i) Let $\mu$ be extremal for $\nu$. By Theorem \ref{thm:my-wiener}(2), the support of $\mu$ is contained in a coset of $\overline{\Gamma_\nu}$ since $\nu$ is good. Because we assume $\overline{\Gamma_\nu}$ to be discrete and countable, the same is true of this coset, hence $\nu$ is a discrete measure. By (ii), $\mu$ is Dirac.
\end{proof}

\section{Application to operator theory}

As an application, we will use our version of Wiener’s lemma to obtain a version of a theorem of Goldstein \cite{goldstein_1996}.

\subsection{Unitary group actions}

As a first step, we will consider the case were the group $H$ acts on a Hilbert space $X$ by contractive operators. Such actions are automatically unitary, so we can use the spectral theory of group actions to obtains measures that we can apply our version of Wiener’s lemma to.

\begin{thm}\label{thm:unitary-group-actions}
Let $X$ be a Hilbert space and $(G, H)$ a pair of dual groups. Let $H$ act on $X$ strongly continuously by unitary operators $T_h$. For $a \in G$ denote by $P_a$ the orthogonal projection onto the subspace $\bigcap_{h \in H} \operatorname{ker} (\langle a, h \rangle \operatorname{I} - T_h)$. Let $\nu$ be a good measure sequence on $H$. Then
\[
    \lim_{N \to \infty} \int_H \left|(T_h x|y)\right|^2 \nu_N(\mathrm{d}h) = \sum_{a \in G} \left|(P_a x | y) \right|^2.
\]
for all $x, y \in X$.
\end{thm}

\begin{proof}
By the spectral theorem applied to the action of $H$ (cf. \cite[Theorem 4.45]{folland_2016}), there exists a unique projection-valued measure $P$ on its dual group $G$ such that the operator $T_h$ associated to any $h \in H$ can be written as
\[
    T_h = \int_G \langle g, h \rangle P(\d g).
\]
By the definition of integrals against projection-valued measures this means that for all $x, y \in X$
\[
    (T_h x|y) = \int_G \langle g, h \rangle \mu_{x, y}(\d g) = \widehat{\mu_{x,y}}(h),
\]
where $\mu_{x,y}$ is the ordinary complex measure given by $\mu_{x,y}(E) = ( P(E)x | y )$. By Theorem \ref{thm:my-wiener} (Wiener’s lemma), we find
\[
    \lim_{N \to \infty} \int_H |(T_h x|y)|^2 \nu_N(\d h) = \lim_{N \to \infty} \Vert\widehat{\mu_{x,y}}\Vert_N^2 = \sum_a |\mu_{x,y}\sing{a}|^2 = \sum_a |(P\sing{a}x|y)|^2,
\]
where the sums are over the atoms of $\mu_{x,y}$.

It remains to show that $P_a = P\sing{a}$. We follow the arguments in \cite[Proposition 5.14]{eisner_farkas_2020}. As both operators are orthogonal projections, we only have to check that the set of fixed points of $P\sing{a}$ is $\bigcap_{h\in H} \operatorname{ker}(\langle a, h \rangle \operatorname{I} - T_h)$. Now $x = P\sing{a} x$ is equivalent to $(P\sing{x}|x) = \Vert x \Vert^2$ by the equality case in the Cauchy--Schwarz inequality. In other words, $\mu_{x,x}\sing{a} = \Vert x \Vert^2$. As $\Vert \mu_{x,x} \Vert = \mu_{x,x}(G) = \Vert x \Vert^2$ (the measure $\mu_{x,x}$ is positive), we get the equivalent equality of measures $\mu_{x,x} = \delta_a \Vert x \Vert^2$. By injectivity of the Fourier transform, we get $\widehat{\mu_{x,x}}(h) = \langle a, h \rangle \Vert x \Vert^2$ (as $\widehat{\delta_a}(h) = \langle a, h \rangle$) for every $h \in H$. Finally, using $(T_h x|x) = \widehat{\mu_{x,x}}(h)$ and equality in the Cauchy--Schwarz inequality again, we arrive at $T_h x = \langle a, h \rangle x$, i.e. $x \in \bigcap_{h\in H} \operatorname{ker}(\langle a, h \rangle \operatorname{I} - T_h)$.
\end{proof}

\subsection{Weak stability}

Considering only contractive group representations excludes many examples because such actions are automatically unitary. To extend our result to actions of some semigroups on Hilbert spaces, we use a variant of a result of Sz.-Nagy and Foiaş that decomposes the Hilbert space on which a given operator acts into an orthogonal sum such that the restrictions of the operator to the subspaces are unitary and weakly stable, respectively. As we want to apply Proposition \ref{prop:koopman-von-neumann}, we need weak stability with respect to filter convergence.

\begin{dfn}
Let $\mathcal{F}$ be a filter on some set $M$ and $(x_m)_{m \in M}$ a family of elements in some Hilbert space $X$. We say that $(x_m)$ \emph{converges weakly} to $x \in X$ with respect to $\mathcal{F}$ if
\[
    \flim{\mathcal{F}}_m (x_m| y) = ( x | y )
\]
for every $y \in X$.
\end{dfn}

\begin{lemma}\label{lem:weakly-convergent-refinement}
Let $\mathcal{F}$ be a filter on $M$ and $(x_m)$ be a bounded family of elements. Then $(x_m)$ converges weakly with respect to some refinement of $\mathcal{F}$.
\end{lemma}

\begin{proof}
Let $\mathcal{U}$ be an ultrafilter refining $\mathcal{F}$. As $(x_m)$ is bounded, so is $(x_m|y)$ for every $y \in X$ and the limit $\flim{\mathcal{U}}_m (x_m | y)$ exists for every $y \in X$. The map $y \mapsto \flim{\mathcal{U}}_m (x_m | y)$ is an conjugate linear and bounded functional on $X$. Hence, there exists $x \in X$ such that $\flim{\mathcal{U}}_m ( x_m | y ) = ( x | y )$ for all $y \in X$.
\end{proof}

\begin{dfn}
Let $M$ be a semigroup and $\mathcal{F}$ a filter $M$. Suppose that $M$ is equipped with a strongly continuous action on a Hilbert space $X$. The action is \emph{weakly stable} (for $\mathcal{F}$) if $(T_m x)_{m \in M}$ converges weakly (with respect to $\mathcal{F}$) to $0$ for every $x \in X$.
\end{dfn}

By choosing the semigroup $\setl{\mathrm{I}, T, T^2, \dots }$ and its Frechét filter, we recover the usual definition of weak stability for the operator $T$.

If $M$ is weakly stable for a filter $\mathcal{F}$, it is also weakly stable for every refinement of $\mathcal{F}$.

\begin{dfn}
We say that a filter $\mathcal{F}$ on an abelian semigroup $M$ is \emph{invariant} if for every set $A \in \mathcal{F}$ and every $n \in M$ the sets $nA = \setb{ na }{ a \in A }$ and $(A : n) = \setb{ m \in M }{ ma \in A }$ are in $\mathcal{F}$ as well.
\end{dfn}

\begin{remark}
If $M$ is a group it suffices to check that $nA \in \mathcal{F}$ for all $n \in M$, $A \in \mathcal{F}$ to show that the filter $\mathcal{F}$ is invariant on $M$ (because $(A : n) = n^{-1} A$). In this case, the filter $\mathcal{F}$ associated to a measure sequence $\nu$ is invariant if the measure sequence is \emph{asymptotically invariant}, i.e. if $n_* \nu_N - \nu_N$ converges strongly to $0$ as $N \to \infty$ for every $n \in M$; in other words, if $\nu_N(n^{-1}A) - \nu_N(A)$ converges to $0$ for every $n \in M$ and every measurable set $A$.

Let $M$ be a subsemigroup of a group $H$ and $\nu$ be an asymptotically invariant measure sequence on $H$ such that the support of each $\nu_N$ is contained in $M$. In this case, we can consider $\nu$ to be a measure sequence on $M$ as well. The filter $\mathcal{F}_{\nu}$ on $M$ associated to this restriction is still invariant.
\end{remark}

If $M$ is an abelian semigroup acting on a Hilbert space $X$, there is also an action of $M$ on $X$ where $m \in M$ acts via the adjoint $T_m^*$ of the operator $T_m$ originally associated to $m$. We call this the action of $M^*$. (If $M$ is not abelian, the action by adjoint operators is an action of the dual semigroup of $M$ instead.)

With this, we get the following version of a theorem of Sz.-Nagy and Foiaş (cf. \cite[Theorem 3.9]{eisner_2010}).

\begin{prop}[Sz.-Nagy--Foiaş decomposition for semigroups]\label{prop:sz-nagy-foias}
Let $M$ be a semigroup with a strongly continuous action by contractions on a Hilbert space $X$ and assume that $M$ is a subsemigroup of an abelian group $H$ such that for every $h \in H$ we have $h \in M$ or $h^{-1} \in M$. Then $X$ is the orthogonal sum of two subspaces $X_1$, $X_2$, both invariant under the action and the adjoint action of $M$, such that
\begin{enumerate}
\item The restriction of the action to $X_1$ is unitary.
\item For every invariant filter $\mathcal{F}$ on $M$ the restriction of the action to $X_2$ is weakly stable with respect to $\mathcal{F}$.
\end{enumerate}
\end{prop}

\begin{proof}

Let
\[
    X_1 := \setb{ x \in X }{ \text{$\Vert T_m x \Vert = \Vert T_m^* x \Vert = \Vert x \Vert$ for all $m \in M$} }.
\]
An easy computation with the Cauchy-Schwarz inequality (cf. \cite[Theorem 3.9]{eisner_2010}) shows that
\[
    X_1 = \setb{ x \in X }{ \text{$T_m^* T_m x = T_m T_m^* x = x$ for all $m \in M$} }.
\]
To see that $X_1$ is $M$-invariant, let $x \in X_1$ and $n \in M$. For every $m \in M$, we have $\Vert T_m T_n x \Vert = \Vert x \Vert = \Vert T_n x \Vert$. We can also write $m = nk$ where $k \in H$. If $k \in M$, then $\Vert T_m^* T_n x \Vert = \Vert T_k^* T_n^* T_n x \Vert = \Vert T_k^* x \Vert = \Vert x \Vert = \Vert T_n x \Vert$. Otherwise we have $k^{-1} \in M$ and we conclude $n = m k^{-1}$, hence
\[
    \Vert x \Vert = \Vert T_{k^{-1}}^* T_m^* T_n x \Vert \leq \Vert T_m^* T_n x \Vert \leq \Vert x \Vert
\]
because the action is contractive. Therefore, $\Vert T_m^* T_n x \Vert = \Vert x \Vert = \Vert T_n x \Vert$. Essentially the same argument shows that $X_1$ is also $M^*$-invariant.

We have to take $X_2 := X_1^\bot$. Let $x \in X_2$ and we have to show that $(T_m x)_{m \in M}$ converges weakly to $0$ with respect to $\mathcal{F}$. We will use the fact that a filter on a compact set converges to an element $x_0$ if and only if every \emph{convergent} refinement of it converges to $x_0$. So suppose that $(T_m x)$ does not converge weakly to $0$; hence there exists $y \in X_2$ such that $(T_m x|y)$ does not converge to $0$ and by the just-stated fact we can choose a refinement $\mathcal{G}$ of $\mathcal{F}$ such that $(T_m x|y)$ converges to something other than $0$ (we get compactness from the boundedness of $(T_m x)$). By Lemma \ref{lem:weakly-convergent-refinement} there exists a refinement $\mathcal{H}$ of $\mathcal{G}$ such that $(T_m x)$ converges weakly to some $x_0$. By the choice of $\mathcal{G}$ and $y$ we have $(x_0|y) \neq 0$, hence $x_0 \neq 0$.

For each $n \in M$, we have
\begin{align*}
    \Vert T_n^* T_n T_m x - T_m x \Vert^2
        &= \Vert T_n^* T_n T_m x \Vert^2 - 2 (T_n^* T_n T_m x | T_m x ) + \Vert T_m x \Vert^2 \\
        &\leq \Vert T_n T_m x \Vert^2 - 2 \Vert T_n T_m x \Vert^2 + \Vert T_m x \Vert^2 
        = \Vert T_m x \Vert^2 - \Vert T_n T_m x \Vert ^2
\end{align*}

The term $\Vert T_m x \Vert^2$ converges to $\inf \setb{ \Vert T_m x \Vert^2 }{ m \in M } =: L$. Indeed, for every $\varepsilon > 0$ there exists $n \in M$ such that $0 \leq \Vert T_n x \Vert^2 - L < \varepsilon$ by the definition of the infimum. Because $M$ acts by contractions, we have $0 \leq \Vert T_m T_n x \Vert^2 - L < \varepsilon$ for every $m \in M$ and the set $nM = \setb{ m n }{ m \in M }$ ($M$ is commutative) belongs to $\mathcal{F}$ because $M$ does.

Because $\mathcal{F}$ is invariant under $A \mapsto (A : n)$ as well, $\Vert T_n T_m x \Vert^2$ converges to the same value. Thus the right-hand side converges to $0$ with respect to $\mathcal{F}$. Therefore, so does the left-hand side.

We have $T_n^* T_n T_m x \to T_n^* T_n x_0$ weakly with respect to $\mathcal{H}$ by continuity of $T_n^* T_n$ and $T_n^* T_n T_m x \to x_0$ with respect to $\mathcal{F}$ (and hence $\mathcal{H}$) by the considerations above; hence $T_n^* T_n x_0 = x_0$ and similarly $T_n T_n^* x_0 = x_0$ for all $n \in M$. Therefore $x_0 \in X_1$. Because $X_2$ is closed and $M$-invariant, we have $x_0 \in X_2$ as well, whence $x_0 = 0$, a contradiction.
\end{proof}

\subsection{Contractive semigroup actions}

We are now ready to combine the results of the previous two sections to obtain a generalization of a theorem of Goldstein \cite[p. 5, Main Theorem]{goldstein_1996}. To retrieve the discrete version of his result, pick $H = \Int$, $M = \Nat$ and the measure sequence $(\mathrm{u}_N)$ on $\Nat$ in the following theorem. For the continuous version, pick $H = \Real$, $M = \Real_{\geq 0}$ and the measure sequence $(\mathrm{u}_N)$ on $\Real$.

\begin{thm}\label{prop:semigroup-goldstein}
Let $(G, H)$ be a pair of dual groups and $M$ a subsemigroup of $H$ such that for any $h \in H$ we have $h \in M$ or $h^{-1} \in M$. Let $M$ act on a Hilbert space $X$ by contractions and such that the action is continuous with respect to the strong operator topology on $\mathcal{L}(X)$. Let $\nu$ be an ergodic and asymptotically invariant measure sequence on $H$ such that the support of $\nu_N$ is contained in $M$ for every $N$. For every $a \in G$ denote by $\mathrm{P}_a$ the orthogonal projection onto $\bigcap_{m \in M} \ker(\langle a, m \rangle \operatorname{I} - T_m)$. Then
\[
    \lim_{N \to \infty} \int_M |(T_m x | y)|^2 \nu_N(\d m) = \sum_{a \in G} |(\mathrm{P}_a x|y)|^2
\]
for all $x, y \in X$.
\end{thm}

\begin{proof}
By Proposition \ref{prop:sz-nagy-foias}, we can write $X$ as an orthogonal direct sum $X_1 \oplus X_2$ of $M$-invariant subspaces where the action of $M$ on $X_1$ is unitary and the action of $M$ on $X_2$ is weakly stable.

Write $x, y \in X$ as $x = x_1 + x_2$ and $y = y_1 + y_2$ where $x_1, y_1 \in X_1$, $x_2, y_2 \in X_2$. We have
\[
    |(T_m x|y)|^2 = |(T_m x_1|y_1)|^2 + 2 \Re (T_m x_1|y_1)\overline{(T_m x_2|y_2)} + |(T_m x_2|y_2)|^2.
\]
By weak stability (and the boundedness of $(T_m x_1|y_1)$), the latter two terms converge to $0$ with respect to $\mathcal{F}_\nu$, so we have
\[
    \lim_{N \to \infty} \int_M \left(2 \Re (T_m x_1|y_1)\overline{(T_m x_2|y_2)} + |(T_m x|y)|^2 \right) \nu_N(\d m) = 0
\]
for all $x, y \in X_2$ by Proposition \ref{prop:koopman-von-neumann}.

For the remaining part on $X_1$, we extend the action of $M$ to an action of $H$: If $h \notin M$, then $h^{-1} \in M$ and we associate to $h$ the operator $(h^{-1})^*$. This gives a well-defined action because the action of $M$ on $X_1$ is unitary and the extended action remains unitary. By Theorem \ref{thm:unitary-group-actions} we have
\[
    \lim_{N \to \infty} \int_H |(T_h x_1 | y_1)|^2 \nu_N(\d h) = \sum_{a \in G} |(P_{1,a} x_1|P_{1,a} y_1)|^2
\]
where $P_{1,a}$ denotes the orthogonal projection onto $\bigcap_{h \in H} \ker (\langle a, h \rangle \operatorname{I} - T_h)$ in $X_1$. As the support of each $\nu_N$ lies in $M$, we can restrict the integration on the left-hand side to $M$ again. The projection does not change when the intersection is taken over $m \in M$ instead because $\langle a, h^{-1} \rangle \operatorname{I} - T_{h^{-1}} = -\langle a, h \rangle^{-1} T_{h^{-1}} (\langle a, h \rangle \operatorname{I} - T_h)$ (where $-\langle a, h \rangle^{-1} T_{h^{-1}}$ is invertible) and hence $\ker (\langle a, h\rangle \operatorname{I} - T_h) = \ker(\langle a, h^{-1} \rangle \operatorname{I} - T_{h^{-1}})$.

Finally, we have to show that the intersection $\bigcap_{m \in M} \ker(\langle a, m \rangle \operatorname{I} - T_m)$ does not change when we interpret $T_m$ as an operator on $X$ instead. By the direct sum decomposition, it suffices to show that this intersection in $0$ when $T_m$ is interpreted as an operator on $X_2$. So let $x \in X_2$ and suppose that $\langle a, m \rangle x - T_m x = 0$ for all $m \in M$; we have to show that $x = 0$. Refine $\mathcal{F}_\nu$ to an ultrafilter $\mathcal{U}$. For every $y \in X_2$ we have
\begin{align*}
    0 
        &= \flim{\mathcal{U}}_m \left( \langle a, m \rangle x - T_m x \,\middle|\, y \right )
        = \flim{\mathcal{U}}_m \langle a, m \rangle (x|y) - \flim{\mathcal{F}_\nu}_m (T_mx|y) \\
        &= (\flim{\mathcal{U}}_m \langle a, m \rangle) (x|y)
\end{align*}
by weak stability. The limit on the right-hand side exists because $\langle a, m \rangle$ is bounded, and its value lies on the unit circle again, hence is non-zero. Thus, $(x|y) = 0$ for every $y \in X_2$, so $x = 0$.
\end{proof}

\printbibliography[heading=bibintoc]

\end{document}